\documentclass{amsart}
\usepackage{amsmath, amsthm,amssymb,amscd,verbatim}
\usepackage[mathscr]{eucal}
\usepackage{enumerate,multicol}
\usepackage{color}

\newtheorem{theorem}{Theorem}[section]
\newtheorem{corollary}[theorem]{Corollary}
\newtheorem{proposition}[theorem]{Proposition}

\theoremstyle{definition}
\newtheorem{definition}[theorem]{Definition}
\newtheorem{remark}[theorem]{Remark}

\newcommand{\abar}{\bar{a}}
\newcommand{\bbar}{\bar{b}}
\newcommand{\cbar}{\bar{c}}

\def\seq{\subseteq}
\def\Def{\operatorname{Def}}
\def\tp{\operatorname{tp}}
\def\Av{\operatorname{Av}}

\def\R{\mathbb{R}}

\def\cL{\mathcal{L}}
\def\cU{\mathcal{U}}
\def\kM{\mathfrak{M}}
\def\miff{\makebox[.5in]{$\Leftrightarrow$}}

\newcommand{\claim}{\hfill$\dashv_{\text{\scriptsize{claim}}}$}

\title[Associativity of the Morley product in NIP theories]{Associativity of the Morley product of invariant measures in NIP theories}

\date{April 16, 2021}

\author[G. Conant]{Gabriel Conant}
\address{DPMMS\\
University of Cambridge\\
Cambridge CB3 0WB\\
 UK}
\email{gconant@maths.cam.ac.uk}

\author[K. Gannon]{Kyle Gannon}

\address{Department of Mathematics\\
University of California, Los Angeles\\
Los Angeles CA 90095\\
 USA}
\email{gannon@math.ucla.edu}

\thanks{GC was partially supported by NSF grant DMS-1855503. KG was partially supported by
NSF research grant DMS-1800806 and NSF CAREER grant DMS-1651321.}

\begin{document}

\begin{abstract}
In light of a gap  found by Krupi\'{n}ski, we give a new proof of associativity for the Morley (or ``nonforking") product of invariant measures in NIP theories.
\end{abstract}

\maketitle

Let $T$ be a complete first-order $\cL$-theory, and fix a sufficiently saturated monster model $\cU$. Given a tuple of variables $x$, we let $\kM_x(\cU)$ denote the space of global Keisler measures on $\Def_x(\cU)$.  Recall that $\kM_x(\cU)$ also corresponds to the space of regular Borel probability measures on the Stone space $S_x(\cU)$ of global types in $x$. 

A measure $\mu\in\kM_x(\cU)$ is \textbf{invariant} if there is a small model $M\prec\cU$ such that, for any $\cL$-formula $\phi(x,y)$ and $b,b'\in\cU^y$, $b\equiv_M b'$ implies $\mu(\phi(x,b))=\mu(\phi(x,b'))$. If $\mu\in\kM_x(\cU)$ is $M$-invariant, and $\phi(x,y)$ is an $\cL_M$-formula, then we have a well-defined function $F^\phi_\mu\colon S_y(M)\to[0,1]$ such that $F^\phi_\mu(q)=\mu(\phi(x,b))$ for some/any $b\models q$. Then $\mu$ is \textbf{Borel-definable} (over $M$) if $F^\phi_\mu$ is a Borel function for any  $\phi(x,y)$. If each of these maps is \emph{continuous}, then $\mu$ is called \textbf{definable} over $M$.

Now fix measures $\mu\in\kM_x(\cU)$ and $\nu\in\kM_y(\cU)$, and assume $\mu$ is Borel-definable (over some small model). The \textbf{Morley product} $\mu\otimes\nu$ (originally defined by Hrushovski and Pillay in \cite{HP}) is constructed as follows. Given an $\cL_{\cU}$-formula $\phi(x,y)$, let $M\prec\cU$ be a small model that contains the parameters in $\phi(x,y)$, and is such that $\mu$ is Borel-definable over $M$. Then
\[
(\mu\otimes\nu)(\phi(x,y))=\int_{S_y(M)}F^\phi_\mu\, d\nu.
\]
One can verify that this does not depend on the choice of $M$, and yields a well-defined Keisler measure in $\kM_{xy}(\cU)$. Moreover, if $\nu$ is $M$-invariant then so is $\mu\otimes\nu$.

Now assume $T$ is NIP. In this case, any $M$-invariant Keisler measure is automatically Borel-definable over $M$ (see \cite[Corollary 4.9]{HP} or \cite[Proposition 7.19]{NIPguide}), and so one can iterate the Morley product. This naturally raises the question of associativity. In \cite[Chapter 7]{NIPguide}, a proof of associativity is sketched, but a gap in the proof was recently found by Krupi\'{n}ski. The purpose of this note is to provide a new proof, which relies heavily on fundamental properties of smooth measures (summarized in Section \ref{sec:prelim}). In Section \ref{sec:main}, we review the proof of associativity for types (which motivates the proof for measures) and then prove the main result. In Section \ref{sec:extra}, we sketch a second proof of associativity, and then we explain the subtlety in the proof sketch from \cite{NIPguide}. This topic is further explored in \cite{CGH} (joint with J. Hanson) where, among other things, we show that the Morley product of Keisler measures can fail to be associative outside of the NIP setting.

\section{Preliminaries on smooth measures}\label{sec:prelim}

Given a tuple $\abar\in (\cU^x)^n$, define $\Av(\abar)\in\kM_x(\cU)$ such that
\[
\textstyle \Av(\abar)(\phi(x))=\frac{1}{n}|\{1\leq i\leq n:\cU\models\phi(a_i)\}|
\]
for any $\cL_{\cU}$-formula $\phi(x)$.
Given $r,s\in\R$ and $\epsilon>0$, we write $r\approx_\epsilon s$ if $|r-s|<\epsilon$.

\begin{definition}
Fix $\mu\in\kM_x(\cU)$ and $M\prec\cU$.
\begin{enumerate}
\item $\mu$ is \textbf{smooth} over $M$ if $\mu$ is the unique measure in $\kM_x(\cU)$ extending $\mu|_M$.
\item $\mu$ is \textbf{finitely approximated} in $M$ if for any $\cL$-formula $\phi(x,y)$ and $\epsilon>0$, there is a tuple $\abar\in (M^x)^n$ such that $\Av(\abar)(\phi(x,b))\approx_\epsilon \mu(\phi(x,b))$ for any $b\in\cU^y$. In this case, we call $\abar$ a \textbf{$(\phi(x,y),\epsilon)$-approximation for $\mu$}.
\end{enumerate}
\end{definition}

Given measures $\mu\in\kM_x(\cU)$ and $\nu\in\kM_y(\cU)$, we say that $\lambda\in\kM_{xy}(\cU)$ is a \textbf{separated amalgam} of $\mu$ and $\nu$ if $\lambda(\phi(x)\wedge\psi(y))=\mu(\phi(x))\nu(\psi(y))$ for any $\cL_{\cU}$-formulas $\phi(x)$ and $\psi(y)$. 

\begin{proposition}\label{prop:smooth}
Suppose $\mu\in\kM_x(\cU)$ is smooth over $M\prec\cU$.
\begin{enumerate}[$(a)$]
\item Let $\phi(x,y)$ be an $L$-formula and fix $\epsilon>0$. Then there are $\cL_M$-formulas $\theta^-_1(x),\ldots,\theta^-_n(x),\theta^+_1(x),\ldots,\theta^+_n(x),\psi_1(y),\ldots,\psi_n(y)$, for some $n\geq 1$, such that:
\begin{enumerate}[$(i)$]
\item the formulas $\psi_1(y),\ldots,\psi_n(y)$ partition $\cU^y$,
\item for all $i\leq n$, if $\cU\models\psi_i(b)$ then $\theta^-_i(x)\seq\phi(x,b)\seq\theta^+_i(x)$, and
\item for all $i\leq n$, $\mu(\theta^+_i(x))-\mu(\theta^-_i(x))<\epsilon$.
\end{enumerate}
Moreover, this implies $\mu$ is definable over $M$.
\item If $\nu\in\kM_y(\cU)$ then $\mu\otimes\nu$ is the unique separated amalgam of $\mu$ and $\nu$ in $\kM_{xy}(\cU)$. In particular, if $\nu$ is Borel-definable, then $\mu\otimes\nu=\nu\otimes\mu$.
\item $\mu$ is finitely approximated in $M$.
\end{enumerate}
\end{proposition}
\begin{proof}
See Lemma 2.3 and Corollary 2.5 of \cite{HPS} for parts $(a)$ and $(b)$. As noted in \cite{HPS}, the symmetry claim in part $(b)$ follows since $\mu\otimes\nu$ and $\nu\otimes\mu$ are both separated amalgams of $\mu$ and $\nu$.
See \cite[Proposition 7.10]{NIPguide} for part $(c)$.
\end{proof}

Part $(c)$ of Proposition \ref{prop:smooth} is also evident from the proof of \cite[Corollary 2.6]{HPS} (see also \cite[Corollary 2.8]{HPS}). The next result is  \cite[Corollary 3.17]{StarBk}, which is stated without proof, and so we take the opportunity here to provide details.

\begin{corollary}\label{cor:smoothprod}
 If $\mu\in\kM_x(\cU)$ and $\nu\in\kM_y(\cU)$ are smooth over $M\prec\cU$, then $\mu\otimes\nu$ is smooth over $M\prec\cU$.
 \end{corollary}
 \begin{proof}
Suppose $\lambda\in\kM_{xy}(\cU)$ is such that $\lambda|_M=(\mu\otimes\nu)|_M$. We want to show that $\lambda=\mu\otimes\nu$. By Proposition \ref{prop:smooth}$(b)$,   it suffices to show that $\lambda$ is a separated amalgam of $\mu$ and $\nu$. 
 So fix $\cL_\cU$-formulas $\phi(x)$ and $\psi(y)$. Fix $\epsilon>0$. By Proposition \ref{prop:smooth}$(a)$, there are $\cL_M$-formulas $\theta^-(x)$, $\theta^+(x)$, $\chi^-(y)$, $\chi^+(y)$ such that:
 \begin{enumerate}[$(i)$]
 \item $\theta^-(x)\seq\phi(x)\seq\theta^+(x)$ and $\chi^-(y)\seq\psi(y)\seq \chi^+(y)$;
 \item $\mu(\theta^+(x))-\mu(\theta^-(x))<\epsilon$ and $\nu(\chi^+(y))-\nu(\chi^-(y))<\epsilon$.
 \end{enumerate}
 (For example, write $\phi(x)$ as $\phi_0(x,b)$ for some $\cL$-formula $\phi_0(x,z)$ and $b\in\cU^z$, and obtain $\theta^-_i(x)$, $\theta^+_i(x)$, $\psi_i(z)$ by applying Proposition \ref{prop:smooth}$(a)$ to $\phi_0(x,z)$ and $\epsilon$. Then choose $i$ such that $\psi_i(b)$ holds, and let $\theta^-(x)=\theta^-_i(x)$ and $\theta^+(x)=\theta^+_i(x)$.)
 
Note that  $\theta^-(x)\wedge\chi^-(y)\seq\phi(x)\wedge\psi(y)\seq\theta^+(x)\wedge\chi^+(y)$. Since $\lambda|_M=(\mu\otimes\nu)|_M$, we have
 \begin{align*}
 \lambda(\theta^-(x)\wedge\chi^-(y)) &=(\mu\otimes\nu)(\theta^-(x)\wedge\chi^-(y))=:r\text{, and}\\
 \lambda(\theta^+(x)\wedge\chi^+(y)) &=(\mu\otimes\nu)(\theta^+(x)\wedge\chi^+(y))=:s.
 \end{align*}
 So $\lambda(\phi(x)\wedge\psi(y)),(\mu\otimes\nu)(\phi(x)\wedge\psi(y))\in [r,s]$. Morever,
 \begin{multline*}
 s-r=\mu(\theta^+(x))\nu(\chi^+(y))-\mu(\theta^-(x))\nu(\chi^-(y))\\
 =\mu(\theta^+(x))\bigg(\nu(\chi^+(y))-\nu(\chi^-(y))\bigg)+\nu(\chi^-(y))\bigg(\mu(\theta^+(x))-\mu(\theta^-(x))\bigg)<2\epsilon.
 \end{multline*}
Therefore $\lambda(\phi(x)\wedge\psi(y))\approx_{2\epsilon}(\mu\otimes\nu)(\phi(x)\wedge\psi(y))$. Since $\epsilon>0$ was arbitrary, we have the desired result.
 \end{proof}

Finally, we recall a main result about NIP theories, namely, the existence of smooth extensions (see \cite[Theorem 3.26]{Keis} or \cite[Proposition 7.9]{NIPguide}).

\begin{theorem}[Keisler]\label{thm:NIPsmooth}
Assume $T$ is NIP. Given $\mu\in\kM_x(\cU)$ and $M\prec\cU$, there is $\nu\in\kM_x(\cU)$ such that $\mu|_M=\nu|_M$ and $\nu$ is smooth over some $N\succ M$.
\end{theorem}

\section{Associativity}\label{sec:main}

Before starting the proof, we briefly recall the argument for associativity of the Morley product for invariant \emph{types} (which holds in any theory). 

Fix global types $p\in S_x(\cU)$ and $q\in S_y(\cU)$, and assume $p$ is invariant over some small model. Then an $\cL_{\cU}$-formula $\phi(x,y)$ is in $p\otimes q$ if and only if $\phi(x,b)\in p$ for some/any $b\models q|_M$, where $M\prec\cU$ contains any parameters in $\phi(x,y)$ and $p$ is $M$-invariant. Now assume $q$ is invariant, and fix a third type $r\in S_z(\cU)$. Let $\phi(x,y,z)$ be an $\cL_{\cU}$-formula, and choose $M\prec\cU$ such that $\phi(x,y,z)$ is over $M$ and $p$ and $q$ are $M$-invariant. Let $c\models r|_M$ and $b\models q|_N$, where $N\prec\cU$ contains $Mc$. Then it is straightforward to show that $(b,c)\models (q\otimes r)|_M$, and thus  $\phi(x,y,z)\in p\otimes(q\otimes r)$ if and only if $\phi(x,b,c)\in p$. On the other hand, since $\phi(x,y,c)$ is over $N$, 
\[
\phi(x,y,z)\in (p\otimes q)\otimes r\miff \phi(x,y,c)\in p\otimes q\miff \phi(x,b,c)\in p.
\]

The proof of associativity for measures follows the same rough strategy, although the individual steps become more intricate. In the above argument, $\tp(b/\cU)$ and $\tp(c/\cU)$ are \emph{isolated} global types extending $q|_N$ and $r|_M$, respectively. For the case of measures, we will replace these by \emph{smooth extensions} (recall that a global type is smooth if and only if it is isolated). Thus we will need to make an NIP assumption to know that such extensions exist (via Theorem \ref{thm:NIPsmooth}). Also, in place of realizations of isolated global types, we will use $\epsilon$-approximations of smooth measures (via Proposition \ref{prop:smooth}$(c)$). Finally, when adapted to measures, several trivial maneuvers in the argument for types require the results in Section \ref{sec:prelim}. For example, the proof above implicitly uses the obvious facts that isolated types commute with any invariant type, and that the product of two isolated types is isolated.

Next, we record a few easy observations.

 \begin{remark}\label{rem:easy}$~$
 \begin{enumerate}[$(a)$]
\item If $\mu\in\kM_x(\cU)$ is Borel definable over $M\prec\cU$, and $\nu,\nu'\in\kM_y(\cU)$ are such that $\nu|_M=\nu'|_M$ then $(\mu\otimes\nu)|_M=(\mu\otimes\nu')|_M$.
\item If $\mu\in\kM_x(\cU)$ is finitely approximated in $M\prec\cU$, $\phi(x,y)$ is an $\cL_M$-formula, and $\abar\in(M^x)^n$ is a $(\phi(x;y),\epsilon)$-approximation for $\mu$ then, for any  $\nu\in\kM_y(\cU)$, 
\[
\textstyle(\mu\otimes\nu)(\phi(x,y))\approx_\epsilon(\Av(\abar)\otimes\nu)(\phi(x,y))=\frac{1}{n}\sum\limits_{i=1}^n\nu(\phi(a_i,y)).
\]
\end{enumerate}
\end{remark}

We now prove the main result.

\begin{theorem}\label{thm:main}
Assume $T$ is NIP, and suppose $\mu\in\kM_x(\cU)$, $\nu\in\kM_y(\cU)$, and $\lambda\in\kM_z(\cU)$. If $\mu$ and $\nu$ are invariant, then $\mu\otimes(\nu\otimes\lambda)=(\mu\otimes\nu)\otimes\lambda$.
\end{theorem}
\begin{proof}
Fix an $\cL_{\cU}$-formula $\phi(x,y,z)$. We want to show 
\[
(\mu\otimes(\nu\otimes\lambda))(\phi(x,y,z))= ((\mu\otimes\nu)\otimes\lambda)(\phi(x,y,z)).
\]
 Let $M\prec\cU$ be a model such that  $\phi(x,y,z)$ is over $M$, and $\mu$ and $\nu$ are $M$-invariant. By Theorem \ref{thm:NIPsmooth}, 
there are $N\succ M$ and $\hat{\lambda}\in\kM_z(\cU)$ such that $\lambda|_M=\hat{\lambda}|_M$ and $\hat{\lambda}$ is smooth over $N$. Similarly, there is $\hat{\nu}\in\kM_y(\cU)$ such that $\nu|_{N}=\hat{\nu}|_{N}$ and $\hat{\nu}$ is smooth over some small model  containing $N$. Note that $\hat{\nu}\otimes\hat{\lambda}$  is smooth by Corollary \ref{cor:smoothprod}. 
\medskip

\noindent\emph{Claim}: $(\hat{\nu}\otimes\hat{\lambda})|_M=(\nu\otimes\lambda)|_M$. 

\noindent\emph{Proof:} 
We have
\[
(\nu\otimes\lambda)|_M=(\nu\otimes\hat{\lambda})|_M=(\hat{\lambda}\otimes\nu)|_M
=(\hat{\lambda}\otimes\hat{\nu})|_M=(\hat{\nu}\otimes\hat{\lambda})|_M,
\]
where the first and third equalities use Remark \ref{rem:easy}$(a)$, while the second and fourth use Proposition \ref{prop:smooth}$(b)$.\claim

\medskip

Now fix some $\epsilon>0$. Let $\phi_1(z;x,y)$, $\phi_2(y;x,z)$, and $\phi_3(y,z;x)$ denote various partitions of the variables in $\phi(x,y,z)$ into object and parameter variables. By Proposition \ref{prop:smooth}$(c)$, we may let $\cbar\in (N^z)^n$ be a $(\phi_1(z;x,y),\epsilon)$-approximation for $\hat{\lambda}$, and let $\bbar\in (\cU^y)^m$ be a $(\phi_2(y;x,z),\epsilon)$-approximation for $\hat{\nu}$.  A straightforward calculation then shows that $(b_i,c_j)_{i\leq m,j\leq n}$ is a $(\phi_3(y,z;x),2\epsilon)$-approximation for $\hat{\nu}\otimes\hat{\lambda}$. Therefore we have
\begin{multline*}
\textstyle(\mu\otimes(\nu\otimes\lambda))(\phi(x,y,z))=(\mu\otimes(\hat{\nu}\otimes\hat{\lambda}))(\phi(x,y,z))\\
\textstyle=((\hat{\nu}\otimes\hat{\lambda})\otimes\mu)(\phi(x,y,z))\approx_{2\epsilon}\frac{1}{mn}\sum\limits_{i=1}^m\sum\limits_{j=1}^n\mu(\phi(x,b_i,c_j)),
\end{multline*}
where the first equality uses Remark \ref{rem:easy}$(a)$ (and the Claim), the second equality uses Proposition \ref{prop:smooth}$(b)$ (and smoothness of $\hat{\nu}\otimes\hat{\lambda}$), and the final approximation uses Remark \ref{rem:easy}$(b)$. On the other hand,
\begin{multline*}
((\mu\otimes\nu)\otimes\lambda)(\phi(x,y,z))=((\mu\otimes\nu)\otimes\hat{\lambda})(\phi(x,y,z))=(\hat{\lambda}\otimes(\mu\otimes\nu))(\phi(x,y,z))\\
\textstyle\approx_\epsilon\frac{1}{n}\sum\limits_{j=1}^n(\mu\otimes\nu)(\phi(x,y,c_j))=\frac{1}{n}\sum\limits_{j=1}^n(\mu\otimes\hat{\nu})(\phi(x,y,c_j))\\
\textstyle =\frac{1}{n}\sum\limits_{j=1}^n(\hat{\nu}\otimes\mu)(\phi(x,y,c_j))\approx_\epsilon\frac{1}{mn}\sum\limits_{j=1}^n\sum\limits_{i=1}^m\mu(\phi(x,b_i,c_j)),
\end{multline*}
where the first and third equalities use Remark \ref{rem:easy}$(a)$, the second and fourth equalities use Proposition \ref{prop:smooth}$(b)$, and the approximations use Remark \ref{rem:easy}$(b)$. 

Altogether, $(\mu\otimes(\nu\otimes\lambda))(\phi(x,y,z))\approx_{4\epsilon} ((\mu\otimes\nu)\otimes\lambda)(\phi(x,y,z))$. Since $\epsilon>0$ was arbitrary, we have $(\mu\otimes(\nu\otimes\lambda))(\phi(x,y,z))= ((\mu\otimes\nu)\otimes\lambda)(\phi(x,y,z))$. 
\end{proof}

\section{Final Remarks}\label{sec:extra}

\subsection{Alternate proof via associativity for smooth measures}
A quicker summary of the proof of associativity for invariant types in arbitrary theories is as follows. Fix $p\in S_x(\cU)$, $q\in S_y(\cU)$, and $r\in S_z(\cU)$, and assume that $p$ and $q$ are invariant. Let $M\prec\cU$ be an arbitrary model such that $p$ and $q$ are $M$-invariant, and choose $c\models r|_M$, $b\models q|_{Mc}$, and $a\models p|_{Mbc}$. Then one easily shows that $((p\otimes q)\otimes r)|_M=(p\otimes(q\otimes r))|_M=\tp(a,b,c/M)$ (see also \cite[Fact 2.20]{NIPguide}). This argument essentially reduces associativity to the case of isolated  types. We can make an analogous reduction in the case of measures, and thus provide a proof of Theorem \ref{thm:main} that avoids explicit use of the fact that smooth measures are finitely approximated (Proposition \ref{prop:smooth}$(c)$). We only need one general fact about associativity in the presence of a smooth measure.

\begin{proposition}\label{prop:smoothassoc}
Suppose $\mu\in\kM_x(\cU)$, $\nu\in\kM_y(\cU)$, and $\lambda\in\kM_z(\cU)$ are such that $\mu$ is smooth, and $\nu$ and $\mu\otimes\nu$ are Borel-definable. Then $(\mu\otimes\nu)\otimes\lambda=\mu\otimes(\nu\otimes\lambda)$.
\end{proposition}
\begin{proof}
Let $\omega=(\mu\otimes\nu)\otimes\lambda$. Since $\mu$ is smooth, it suffices by Proposition \ref{prop:smooth}$(b)$ to show that $\omega$ is a separated amalgam of $\mu$ and $\nu\otimes\lambda$. So fix $\cL_{\cU}$-formulas $\phi(x)$ and $\psi(y,z)$, and let $\theta(x,y,z):=\phi(x)\wedge\psi(y,z)$. Fix $M\prec\cU$ such that $\theta(x,y,z)$ is over $M$, and $\mu$, $\nu$, and $\mu\otimes\nu$ are Borel-definable over $M$.  
Then
\[
\omega(\theta(x,y,z))=\int_{S_z(M)}F^\theta_{\mu\otimes\nu}\, d\lambda=\mu(\phi(x))\int_{S_z(M)}F^\psi_\nu\,d\lambda=\mu(\phi(x))(\nu\otimes\lambda)(\psi(y,z)),
\]
as desired. 
\end{proof}

\begin{remark}
In fact, the conclusion of the previous proposition holds even under the weaker assumption that $\mu$ is only definable (but the proof requires more work; see \cite[Theorem 2.18]{CGH}). In the same result from \cite{CGH}, it is also shown that if $\mu$ is definable and $\nu$ is Borel definable, then $\mu\otimes\nu$ is Borel definable. Thus the Borel definability assumption on $\mu\otimes\nu$ in Proposition \ref{prop:smoothassoc} is superfluous.
\end{remark}

Now assume $T$ is NIP, and fix measures $\mu\in\kM_x(\cU)$, $\nu\in\kM_y(\cU)$, and $\lambda\in\kM_z(\cU)$, with $\mu$ and $\nu$ invariant. Fix an $\cL_{\cU}$-formula $\phi(x,y,z)$. Let $M\prec\cU$ be a small model, such that $\phi(x,y,z)$ is over $M$, and $\mu$ and $\nu$ are invariant over $M$. By Theorem \ref{thm:NIPsmooth}, there are models $N_1\succ N_0\succ M$ and measures $\hat{\mu}\in\kM_x(\cU)$, $\hat{\nu}\in\kM_y(\cU)$, $\hat{\lambda}\in\kM_z(\cU)$ such that $\hat{\lambda}|_M=\lambda|_M$, $\hat{\nu}|_{N_0}=\nu|_{N_0}$, $\hat{\mu}|_{N_1}=\mu|_{N_1}$, $\hat{\lambda}$ is smooth over $N_0$, $\hat{\nu}$ is smooth over $N_1$, and $\hat{\mu}$ is smooth over some model containing $N_1$. By Corollary \ref{cor:smoothprod}, $\hat{\nu}\otimes\hat{\lambda}$ is smooth over $N_1$. Now, using Remark \ref{rem:easy}$(a)$ and Proposition \ref{prop:smooth}$(b)$ (as in the Claim in the proof of Theorem \ref{thm:main}), we have
\[
(\hat{\mu}\otimes\hat{\nu})|_{N_0}=(\mu\otimes\nu)|_{N_0}\makebox[.5in]{and}(\hat{\nu}\otimes\hat{\lambda})|_M=(\nu\otimes\lambda)|_M.
\]
So using Remark \ref{rem:easy}$(a)$, Proposition \ref{prop:smooth}$(b)$, and Proposition \ref{prop:smoothassoc}, we have
\begin{multline*}
((\mu\otimes\nu)\otimes\lambda)(\phi(x,y,z)) = ((\mu\otimes\nu)\otimes\hat{\lambda})(\phi(x,y,z))\\
= (\hat{\lambda}\otimes(\mu\otimes\nu))(\phi(x,y,z)) = (\hat{\lambda}\otimes(\hat{\mu}\otimes\hat{\nu}))(\phi(x,y,z))\\
= ((\hat{\mu}\otimes\hat{\nu})\otimes\hat{\lambda})(\phi(x,y,z)) = (\hat{\mu}\otimes(\hat{\nu}\otimes\hat{\lambda}))(\phi(x,y,z))\\
= ((\hat{\nu}\otimes\hat{\lambda})\otimes\hat{\mu})(\phi(x,y,z)) =((\hat{\nu}\otimes\hat{\lambda})\otimes\mu)(\phi(x,y,z))\\
= (\mu\otimes(\hat{\nu}\otimes\hat{\lambda}))(\phi(x,y,z)) = (\mu\otimes(\nu\otimes\lambda))(\phi(x,y,z)).
\end{multline*} 

Note that in the previous argument, we only needed a weaker version of Proposition \ref{prop:smoothassoc} in which all three measures are smooth. Since smooth measures are definable (by Proposition \ref{prop:smooth}$(a)$), one could instead use associativity for definable measures, which has a short and fairly elementary proof (see \cite[Proposition 2.6]{CoGa}).

\subsection{Product measures of Borel sets}
We finish with a discussion of  the proof sketch of associativity in \cite{NIPguide}. The goal is to identify a certain unexpected subtlety (namely, the potential failure of equation $(\dagger\dagger)$ below) that arises when fleshing out the details. 

Assume $T$ is NIP. Fix invariant measures $\mu\in\kM_x(\cU)$ and $\nu\in\kM_y(\cU)$, and an arbitrary measure $\lambda\in\kM_z(\cU)$. Toward proving $\mu\otimes(\nu\otimes\lambda)=(\mu\otimes\nu)\otimes\lambda$, it suffices by  \cite[Proposition 7.11]{NIPguide} to assume that $\mu$ is an invariant type $p\in S_x(\cU)$. Now fix a formula $\phi(x,y,z)$. We want to show 
\begin{equation*}
(p\otimes(\nu\otimes\lambda))(\phi(x,y,z))=((p\otimes\nu)\otimes\lambda)(\phi(x,y,z)).\tag{$\dagger$}
\end{equation*}

Choose $M\prec\cU$ containing all parameters in $\phi$, such that $p$ and $\nu$ are invariant over $M$. Since $p$ is then Borel-definable over $M$, the set 
\[
B(y,z)=\{s\in S_{yz}(M):\phi(x,b,c)\in p\text{ for some/any $(b,c)\models s$}\}
\]
is Borel. By direct computation, we have $(p\otimes(\nu\otimes\lambda))(\phi(x,y,z))=(\nu\otimes\lambda)(B)$. On the other hand,  $((p\otimes\nu)\otimes\lambda)(\phi(x,y,z))=\int_{S_z(M)}F^\phi_{p\otimes\nu}\, d\lambda$ by definition of the Morley product. 
In order to compare these two values, we need to relate $F^\phi_{p\otimes\nu}$ to the set $B$. In particular, given some $c\in \cU^z$, we can define the  ``fiber" 
\[
B(y,c)=\{q\in S_y(\cU):\tp(b,c/M)\in B(y,z)\text{ for some/any }b\models q|_{Mc}\},
\]
which itself is Borel (e.g., apply Borel-definability of $p$ to the formula $\phi(x,y,c)$). So we have a well-defined function  $F^B_\nu\colon S_z(M)\to [0,1]$ such that $F^B_\nu(r)=\nu(B(y,c))$ for some/any $c\models r$. 
\medskip

\noindent\emph{Claim.} $F^B_\nu=F^{\phi}_{p\otimes\nu}$.

\noindent\emph{Proof.} Fix $r\in S_z(M)$, $c\models r$, and $N\prec\cU$ containing $Mc$. Let $\rho\colon S_y(\cU)\to S_y(N)$ be the restriction map. Note first that $\nu|_N$ is the pushforward of $\nu$ along $\rho$ (see, e.g.,  \cite[Remark 1.2]{CGH}). Moreover, $\rho(B(y,c))$ is Borel since it is precisely the preimage of $\{1\}$ under $F^{\phi(x,y,c)}_p$ (and $p$ is Borel-definable). Finally, $B(y,c)$ is ``$N$-invariant" in the sense that $\rho^{\text{-}1}(\rho(B(y,c)))=B(y,c)$. Combining these observations, we have
\begin{multline*}
F^\phi_{p\otimes\nu}(r)=(p\otimes\nu)(\phi(x,y,c))=\int_{S_y(N)}F^{\phi(x,y,c)}_p\, d\nu=\nu|_N(\rho(B(y,c))\\
 =\nu|_N(\rho^{\text{-}1}(\rho(B(y,c)))) =\nu(B(y,c))=F^B_\nu(r).\qed
\end{multline*}

\medskip 

Altogether, $(\dagger)$ can be rewritten as the following equality:
\begin{equation*}
(\nu\otimes\lambda)(B)=\int_{S_z(M)}F^B_\nu\,d\lambda.\tag{$\dagger\dagger$}
\end{equation*}
In the sketch from \cite{NIPguide}, this equality is asserted without further detail. Note that if $B$ is \emph{clopen} (i.e., a formula), then $(\dagger\dagger)$ is precisely the definition of the Morley product. However, as noted by Krupi\'{n}ski, this does not automatically imply that the same equality should hold for arbitrary Borel sets. For example, if $B$ is only open, then a direct approach to $(\dagger\dagger)$ requires one to interchange an integral with a limit over some net of functions (via regularity of $\nu\otimes\lambda$ as a measure on $S_{yz}(\cU)$). In the setting of abstract integration, this is not always possible. 
On the other hand, Theorem \ref{thm:main} implies (\emph{a posteriori}) that $(\dagger\dagger)$ does indeed hold when $T$ is NIP and $B$ is a Borel set arising as the $\phi$-definition of an invariant global type $p\in S_x(\cU)$, for some formula $\phi(x,y,z)$. 

The obvious question at this point is what happens when $T$ is not NIP.
In fact, several issues similar to those identified in the previous discussion were also noted by the authors while considering Borel-definable measures in \emph{arbitrary} theories during the preparation of \cite{CoGa}. Ultimately, these obstacles were avoided  by narrowing the focus to definable measures, which sufficed for the main results (see \cite[Remark 2.11]{CoGa}). Following an earlier draft of this note, we resumed the general investigation of Borel-definable measures in \cite{CGH} (joint with J. Hanson).  In Section 3.3 of \cite{CGH}, we work in the theory of the random ternary relation (which is \emph{not} NIP), and construct $p\in S_1(\cU)$, $q\in S_1(\cU)$, and $\lambda\in \kM_1(\cU)$ such that $p\otimes(q\otimes \lambda)\neq (p\otimes q)\otimes\lambda$, even though all measures involved are Borel-definable. Using the same steps as above, this can be rewritten as an example where both sides of  $(\dagger\dagger)$ are well-defined, but equality does not hold.

\subsection*{Acknowledgments} The authors would like to thank Krzysztof Krupi\'{n}ski and Anand Pillay for helpful comments. The second author would also like to thank Sergei Starchenko for many important ideological conversations.


\begin{thebibliography}{1}

\bibitem{CoGa}
Gabriel Conant and Kyle Gannon, \emph{Remarks on generic stability in
  independent theories}, Ann. Pure Appl. Logic \textbf{171} (2020), no.~2,
  102736, 20. \MR{4033642}
  
  \bibitem{CGH}
  Gabriel Conant, Kyle Gannon, and James Hanson, \emph{Keisler measures in the wild}, arXiv 2103.09137, 2021.

\bibitem{HP}
Ehud Hrushovski and Anand Pillay, \emph{On {NIP} and invariant measures}, J.
  Eur. Math. Soc. (JEMS) \textbf{13} (2011), no.~4, 1005--1061. \MR{2800483}

\bibitem{HPS}
Ehud Hrushovski, Anand Pillay, and Pierre Simon, \emph{Generically stable and
  smooth measures in {NIP} theories}, Trans. Amer. Math. Soc. \textbf{365}
  (2013), no.~5, 2341--2366. \MR{3020101}

\bibitem{Keis}
H.~Jerome Keisler, \emph{Measures and forking}, Ann. Pure Appl. Logic
  \textbf{34} (1987), no.~2, 119--169. \MR{890599}

\bibitem{NIPguide}
Pierre Simon, \emph{A guide to {NIP} theories}, Lecture Notes in Logic,
  vol.~44, Association for Symbolic Logic, Chicago, IL; Cambridge Scientific
  Publishers, Cambridge, 2015. \MR{3560428}
  
\bibitem{StarBk}
Sergei Starchenko, \emph{N{IP}, {K}eisler measures and combinatorics}, no. 390,
  2017, S\'{e}minaire Bourbaki. Vol. 2015/2016. Expos\'{e}s 1104--1119,
  pp.~Exp. No. 1114, 303--334. \MR{3666030}

\end{thebibliography}
\end{document}